\documentclass[11pt]{article}
\usepackage{amsmath, amssymb, theorem, latexsym, epsfig}
 \usepackage{ifpdf}
\numberwithin{equation}{section}

\theoremstyle{plain}
\theorembodyfont{\itshape}
\newtheorem{theorem}{Theorem}[section]
\newtheorem{proposition}[theorem]{Proposition}
\newtheorem{lemma}[theorem]{Lemma}
\newtheorem{corollary}[theorem]{Corollary}

\theorembodyfont{\rmfamily}
\newtheorem{definition}[theorem]{Definition}

\newtheorem{example}[theorem]{Example}
\newtheorem{remark}[theorem]{Remark}

\newenvironment{proof}{{\noindent \textbf{Proof}\,\,}}{\hspace*{\fill}$\Box$\medskip}

\def\wt#1{\widetilde#1}
\def\rr{\mathbb R}
\def\hh{\mathbb H}

\def\mcr{\mathcal R}

 \def\rp{\mathbb{RP}}

\def\La{\Lambda}

\def\mcr{\mathcal R}

\def\mcn{\mathcal N}

\def\Xi{\mathcal Z}
\def\mcl{\mathcal L}

\def\wh#1{\widehat#1}

\title{On Hamiltonian projective billiards on boundaries of products of convex bodies}
\author{Alexey Glutsyuk\thanks{CNRS, UMR 5669 (UMPA, ENS de Lyon), France. E-mail: 
aglutsyu@ens-lyon.fr} \thanks{HSE University, Moscow, Russia} \thanks{Higher School of Modern Mathematics MIPT, 1 Klimentovskiy per., Moscow, Russia}}
\begin{document}
\maketitle
\vspace{-0.3cm}
{\it To Dmitry Valerievich Treschev  on the occasion of his 60-th  birthday.}

{\it To  Sergei Vladimirovich Bolotin  on the occasion of his 70-th  birthday.}
\begin{abstract} 
Let $K\subset\rr^n_q$, $T\subset\rr^n_p$ be two bounded strictly convex bodies (open subsets) with $C^6$-smooth boundaries. We consider the product $\overline K\times\overline T\subset\rr^{2n}_{q,p}$ equipped with the standard symplectic form $\omega=\sum_{j=1}^ndq_j\wedge dp_j$. The  {\it $(K,T)$-billiard orbits} are continuous curves in 
the boundary $\partial(K\times T)$ whose intersections with the open dense subset 
$(K\times\partial T)\cup(\partial K\times T)$ are tangent to the  characteristic line field given by 
kernels of the restrictions of the symplectic form $\omega$ to the tangent spaces to the boundary. 
For every $(q,p)\in K\times \partial T$ the characteristic line  in $T_{(q,p)}\rr^{2n}$ is directed by the vector 
$(\vec n(p),0)$, where $\vec n(p)$ is the exterior normal to $T_p\partial T$, and similar statement holds for $(q,p)\in\partial K\times T$. The projection of each $(K,T)$-billiard orbit to $K$ is 
an orbit of the so-called $T$-billiard in $K$. In the case, when $T$ is centrally-symmetric, this  is the billiard in $\rr^n_q$ equipped with  Minkowski Finsler structure 
"dual to $T$", with Finsler reflection law introduced in a joint paper by S.Tabachnikov and E.Gutkin in 2002.  Studying $(K,T)$-billiard orbits is closely related to  C.Viterbo's Symplectic Isoperimetric Conjecture (recently disproved by P.Haim-Kislev and Y.Ostrover) and the famous Mahler Conjecture in convex geometry. 
We study the special case, when the $T$-billiard reflection law is the projective  law introduced by S.Tabachnikov, i.e., given by projective involutions of the projectivized tangent spaces $T_q\rr^n$, $q\in\partial K$. We show that this happens, if and only if  $T$ is an ellipsoid, or equivalently, if all the $T$-billiards are simultaneously affine equivalent to  Euclidean billiards.
As an application, we deduce analogous results for Finsler billiards.

\end{abstract}
\tableofcontents

\section{Finsler and $(K,T)$-billiards with  projective reflection law. Introduction and main results}

Let $K\subset\rr^n_{q_1,\dots,q_n}$ and $T\subset\rr^n_{p_1,\dots,p_n}$ be two bounded 
strictly convex open subsets  with $C^6$-smooth boundaries. We consider their 
product $K\times T$ in the standard symplectic space $\rr^{2n}$ equipped with the symplectic form 
$\omega:=\sum_{j=1}^ndq_j\wedge dp_j$. The open and dense subset $(K\times\partial T)\cup(\partial K\times T)\subset\partial(K\times T)$ carries the characteristic line field: the field of kernels of the restriction to 
$\partial(K\times T)$ of the symplectic form $\omega$. It is well-known that the 
characteristic line field is directed by the vectors $(\vec n(p),0)\in T_{q,p}\rr^2$ for $(q,p)\in K\times\partial T$,  respectively, $(0,-\vec n(q))$ for $(q,p)\in \partial K\times T$, where $\vec n(q)$, $\vec n(p)$ are the exterior normal vector fields on $\partial K$ and $\partial T$ respectively:  see \cite{ako1, ao2} and the next paragraph. 
The above vectors induce the canonical orientation of the characteristic line field that can be also defined as follows. Let $H(p,q)$ be a non-negative Hamiltonian function on $\rr^{2n}$ such that $K\times T=\{ H<1\}$. We consider that $H$ satisfies the following additional condition:

{\it -  for every fixed $p\in T$  the restriction to 
$\overline K\times\{ p\}$ of the function $H$ has non-zero gradient at the points of the boundary 
$\partial K\times\{ p\}$, and the same holds for $K$ interchanged with $T$.} 

The  corresponding Hamiltonian vector field is tangent to the boundary $\partial(K\times T)$ and directs the above characteristic line field. Its restriction to 
$K\times \partial T$ coincides with the field $(\frac{\partial H}{\partial p},0)$, which is clearly co-directed with  the field $(\vec n(p),0)$. Similarly, its restriction to 
$\partial K\times  T$ coincides with the field $(0, -\frac{\partial H}{\partial q})$, which is  co-directed with  the field $(0,-\vec n(q))$.

\begin{definition} 
A continuous oriented curve  in $\partial(K\times T)$ whose intersection with the union $(K\times\partial T)
\cup(\partial K\times T)$  is tangent and co-directed with the above characteristic line field  
is called an {\it orbit of the $(K,T)$-billiard.}
\end{definition}

It is well-known, see \cite{ako1, ao2}, that  the dynamics of the above $(K,T)$-billiard (or equivalently, of 
the above Hamiltonian flow) is indeed 
"billiard-like". Namely, the Hamiltonian flow moves each point $(q,p_0)\in K\times\partial T$ so that $q$ moves with a velocity parallel and co-directed with 
the exterior normal vector $\vec n(p_0)$ to $\partial T$ at $p_0$. Let $L_0$ denote the 
movement line of the point $q$, oriented by the velocity. As $q$ reaches the boundary $\partial K$, say, at some point $q_1$, the point $p$ starts to move with  a velocity parallel and co-directed with the interior normal vector to $\partial K$ at $q_1$, and $q_1$ remains fixed. This latter movement continues until $p$ hits the boundary $\partial T$, say, at some point $p_1$. Afterwards $q=q_1$ starts 
to move  with  a velocity  parallel and co-directed to the exterior normal to $\partial T$ at 
$p_1$, etc. Let $L_1$ denote the oriented line of the latter movement of the point $q$. The line $L_0$ is directed out of $K$ at $q_1$, and $L_1$ is directed inside $K$ at $q_1$. See Figure 1. 
The {\it reflection map} 
$$R_{K,q_1}:L_0\mapsto L_1$$
 is a diffeomorphism from the space of oriented lines through $q_1$ directed outside $K$ at $q_1$ 
 to the space of oriented lines through $q_1$ directed inside $K$ at $q_1$. 
 The family of thus defined reflection maps $R_{K,q_1}$ parametrized by the bouncing point $q_1$ will be briefly denoted $R_K$. For every $q_1\in\partial K$ the map $R_{K,q_1}$  is the restriction to the outward directed lines of a diffeomorphic involution $S^{n-1}\to S^{n-1}$ of the space of oriented lines through $q_1$ 
 (identified with the sphere $S^{n-1}$) that fixes the lines lying in the tangent plane $T_{q_1}\partial K$. 
 The projection $\pi_q:\rr^n_q\times\rr^n_p\to\rr^n_q$ sends each $(K,T)$-billiard orbit to an orbit of the so-called 
 $T$-billiard in $K$, see \cite{ako1, ao2, gt} and the next definition. 
 
 \begin{definition} \label{deftbil} 
 Let $K,T\subset\rr^n$ be bounded strictly convex bodies. The {\it $T$-billiard} in $K$ is the dynamical system  on the space of  oriented lines $L$ intersecting $K$ defined as follows. For every $L$ take 
 its last point $q$ of intersection with the boundary $\partial K$ in the sense of orientation. 
 The {\it image of the line $L$ under the 
 $T$-billiard map} is its image under the reflection map $R_{K,q}$. 
 \end{definition}
 
 \begin{figure}[ht]
  \begin{center}
   \epsfig{file=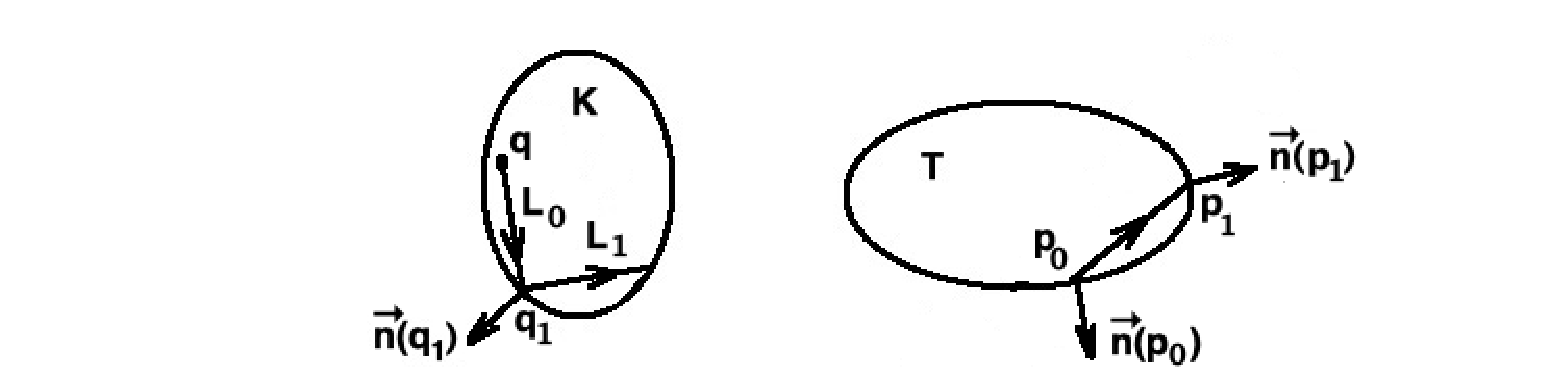, width=35em}
    \caption{The $(K,T)$-billiard orbits.}
    \label{fig:0}
  \end{center}
\end{figure}

\begin{remark} In the case, when the body $T$ is centrally-symmetric, the $T$-billiard coincides with the Minkowski Finsler billiard introduced in a joint paper by E.Gutkin and S.Tabachnikov \cite{gt} where the Finsler unit ball is  the dual to the body $T$. See the corresponding material in Subsection 1.2. 
\end{remark}

The next proposition is a well-known fact, see \cite[p. 2011]{ako1}, which follows from definition. 

\begin{proposition} \label{pro-ball} Let $T$ be a ball. Then the $T$-billiard is the usual Euclidean billiard. 
Namely, for every $q\in\partial K$ the reflection $R_{K,q}$  is the standard reflection in the billiard in  $K$ at $q$ defined by the usual reflection law: the angle of incidence is equal to the angle of reflection. 
\end{proposition}

Studying $(K,T)$-billiard orbits is closely related to  C.Viterbo's Symplectic Isoperimetric Conjecture, 
see \cite{viterbo} and \cite[p. 2006, conjecture 1.2]{ako1}, which 
states that for every convex body $\Sigma$ in the standard symplectic space $\rr^{2n}$ and for every symplectic capacity the ratio of capacities of the body $\Sigma$ and the unit ball is no greater than the ratio of their volumes  in power $\frac1n$. Namely, it was shown in \cite{hz}, see also \cite{ao2}, that the Hofer-Zehnder capacity of every convex bounded domain in $\rr^{2n}$ is equal to the minimal action of closed orbit of characteristic line field on its boundary. In particular, this holds for closed $(K,T)$-billiard orbits. 
As was shown in a joint paper 
by S.Artstein-Avidan, R.Karasev and Y.Ostrover, see \cite[theorem 1.6]{ako1},
the Viterbo's Conjecture implies the famous Mahler Conjecture, which states that the product of every 
convex body $K\subset\rr^n$ with its dual has volume no less than $\frac{4^n}{n!}$. 
The Mahler Conjecture in three dimensions was proved in \cite{is}. 

Very recently a counterexample to the Viterbo's Conjecture was constructed in a joint paper by 
P. Haim-Kislev and Y.Ostrover \cite{hko} using dynamics of $(K,T)$-billiards.

In \cite{tabpr} Sergei Tabachnikov introduced the projective billiards, which are common generalization 
of billiards on all the space forms of constant curvature, see the next definition. 

The goal of the present paper is to describe those convex bodies $T$ for which the corresponding $T$-billiard  is a projective billiard. Our main result shows that this holds 
only in the case, when $T$ is an ellipsoid (or more generally, a quadric, if $T$ is not necessarily bounded). 

The main results are stated in Subsection 1.1. As an application, in Subsection 1.2 we state and deduce similar results on Finsler billiards. 

\subsection{Main results. The $T$-billiards with projective reflection law}

\begin{definition} see \cite{tabpr}. A {\it projective billiard} is a smooth hypersurface $C\subset\rr^n$ equipped with a transversal line field $\mcn$. 
 For every $Q\in C$ the {\it projective billiard reflection involution} at $Q$ acts on the space of oriented lines through $Q$ as the affine involution 
 $\rr^n\to\rr^n$ that fixes the points of the tangent line to $C$ at $Q$, preserves the line $\mcn(Q)$ and acts on $\mcn(Q)$ as central symmetry 
 with respect to the point\footnote{In other words, two (non-oriented) lines $a$, $b$ through $Q$  are permuted by reflection at $Q$, if  and only if $\mcn(Q)$, $a$, $b$ lie in one two-dimensional subspace $\Pi$ (with origin at $Q$) and the quadruple of lines $\ell:=\Pi\cap T_QC$, $\mcn(Q)$, $a$, $b$ is harmonic: there exists a projective  involution of the projectivized subspace $\Pi$ 
 that fixes $\ell$, $\mcn(Q)$ and permutes $a$, $b$.} $Q$. 
 In the case, when $C$ is a strictly convex closed hypersurface,  the {\it projective billiard map} 
  acts on the {\it phase cylinder:} 
 the space of oriented lines intersecting $C$. It is not an involution. It sends an oriented line to its image under the above reflection involution at its last point $Q$ 
 of intersection with $C$ in the sense of orientation. See Figure 2.
 \end{definition} 
 \begin{figure}[ht]
  \begin{center}
   \epsfig{file=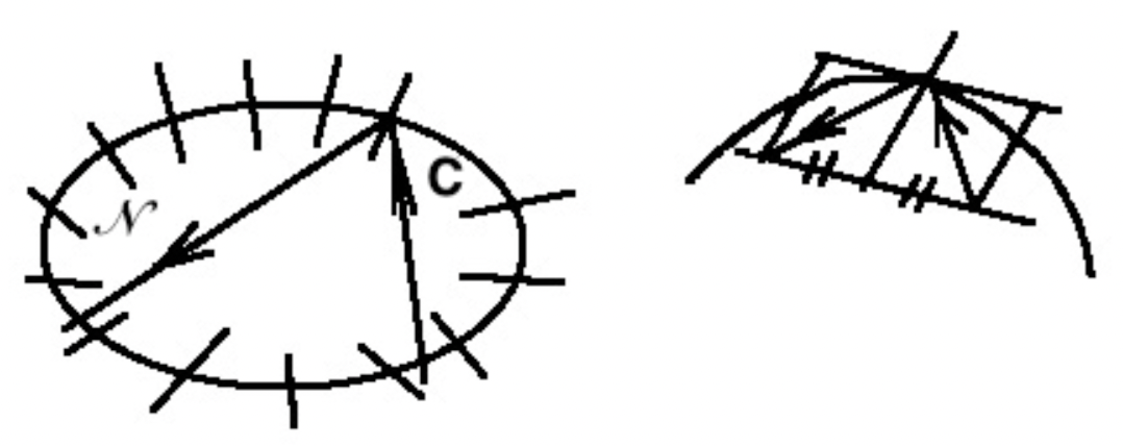, width=28em}
    \caption{The projective billiard reflection.}
    \label{fig:0}
  \end{center}
\end{figure}
\begin{example} see \cite{tabpr}. \label{exconst} 
A usual Euclidean planar billiard is a projective billiard with transversal line field being the 
normal line field. Each space form of constant curvature, i.e., $\rr^n$, $S^n$, $\hh^n$, is realized as 
a hypersurface $\Sigma\subset\rr^{n+1}_{x_1,\dots,x_{n+1}}$ equipped with an appropriate quadratic form $\Psi$, the Riemannian metric on the hypersurface being induced by the quadratic form: 
$$\rr^n=\{ x_{n+1}=1\}, \ \Psi:=\sum_{j=1}^nx_j^2;$$
$$S^n=\{ \sum_{j=1}^{n+1}x_j^2=1\}, \ \Psi=\sum_{j=1}^{n+1}x_j^2;$$
$$\hh^n=\{\sum_{j=1}^{n}x_j^2-x_{n+1}^2=-1\}, \ \Psi=\sum_{j=1}^{n}x_j^2-x_{n+1}^2.$$
 The tautological projection $\pi:\rr^{n+1}\setminus\{0\}\to\rp^n$ sends 
$\Sigma_+:=\Sigma\cap\{ x_{n+1}>0\}$ to the affine chart 
$\rr^n=\{ x_{n+1}=1\}\subset\rp^n_{[x_1:\dots:x_{n+1}]}$. Each Riemannian billiard on $\Sigma$ bounded by a hypersurface $\alpha\subset\Sigma$ is projected to a 
projective billiard on the hypersurface $C=\pi(\alpha)$ with the transversal line field $\mcn$ 
being the image of the normal line field to $\alpha$ by the differential $d\pi$. The projection $\pi$ 
sends each orbit of the Riemannian billiard on $\alpha$ to an orbit of the corresponding projective 
billiard.
\end{example}

\begin{definition} A projective billiard on a hypersurface $C\subset\rr^n$ is said to be {\it affinely Euclidean,} if there exists an affine transformation $F:\rr^n\to\rr^n$  transforming it into a usual billiard, with usual reflection of 
lines from the image $F(C)$: this holds if and only if the image of the transversal line field $\mcn$ 
on $C$ under the differential $dF$ is the normal line field to $F(C)$.
\end{definition}

\begin{proposition} \label{cell} 1) Let $T\subset\rr^n$ be an ellipsoid, and let $K\subset\rr^n$ be a convex bounded open subset with a smooth boundary.  Then the corresponding $T$-billiard reflection transformation $R_K$ is a reflection defined by a projective billiard structure on $\partial K$. 

2) The latter projective billiard structure is affinely Euclidean.
\end{proposition}     
\begin{proof} Let us choose orthogonal coordinates $p_1,\dots,p_n$ in which $T$ is an ellipse 
$\sum_{j=1}^n\frac{p_j^2}{b_j^2}=1$. Let us complete $p_j$ by affine coordinates $q_1,\dots,q_n$ 
on $K$ so that $(q_1,\dots,q_n, p_1,\dots,p_n)$ are symplectic coordinates. The variable change 
\begin{equation}\wt p_j=\frac{p_j}{b_j}, \ \wt q_j=b_j q_j\label{rescal}\end{equation}
is symplectic and transforms $K\times T$ to the product of the rescaled image $\wt K$ of the 
set $K$ and the unit ball in the $\wt p$-coordinates. It conjugates the Hamiltonian flows with a 
Hamiltonian function $H$ as at the beginning of the paper in the coordinates $(p,q)$ and $(\wt p,\wt q)$. In the 
coordinates $(\wt q,\wt p)$  the points $\wt q$ move along usual billiard orbits in $\wt K$, by Proposition \ref{pro-ball}. Therefore, 
the above rescaling sends the  movements of the points $q\in K$ with reflections $I_K$ to billiard 
orbits in $\wt K$. Thus, the rescaling in the $q$-coordinates conjugates the 
reflection $I_K$  at each point $q\in\partial K$ with the standard billiard 
reflection $I_{\wt K}$ at $\wt q$. Hence, $I_K$ is a projective billiard reflection: it is defined by the transversal line field 
$\mcn$ on $\partial K$ that is the pullback of the normal field on $\wt K$ under the 
rescaling (\ref{rescal}) of the $q$-coordinates. The corollary is proved.
\end{proof}
\begin{corollary} \label{cquad} Statement 1) of Proposition \ref{cell} is valid in the case when $T$ is 
bounded by an arbitrary (maybe unbounded) quadric.
\end{corollary}

Corollary \ref{cquad} follows from Statement 1) of Proposition  \ref{cell} by analytic extension of the projectivity of reflection map $R_{K,q}$ as $\partial T$ varies analytically in a family of complex quadrics.

We prove the next theorem, which  is a converse statement. 

\begin{theorem} \label{prodim2} Let in a  $(K,T)$-billiard  the reflection $I_K$  be 
a projective billiard reflection. Let in addition the boundary $\partial T$ be 
$C^6$-smooth, closed and have positive definite second fundamental quadratic form everywhere. Then $\partial T$ is an ellipsoid.
\end{theorem}

Theorem \ref{prodim2} has an equivalent reformulation stated below.  To state it, consider an arbitrary 
convex bounded domain $T$ in $\rr^n$ whose boundary is $C^6$-smooth and has positive definite second fundamental quadratic form. 
Let us identify $\partial T$ with the unit sphere $S^{n-1}$ via the Gauss map sending a point $p\in\partial T$ 
to the exterior normal vector to $T$ at $p$. 
For a given class $\mcl$ of parallel lines consider the involution 
$\mcr_\mcl: S^{n-1}\to S^{n-1}$ acting as follows. 
Take an arbitrary  line $\La\in\mcl$ intersecting $\partial T$, let $A$ and $B$ be their two intersection points (confluenting to one point in the case of tangency). Let $\vec n(A)$, $\vec n(B)$ be 
respectively the corresponding exterior normal vectors to $\partial T$. By definition,
 \begin{equation} \mcr_\mcl(\vec n(A))=\vec n(B).\label{jmcl}\end{equation}
 The map $\mcr_\mcl$ fixes vectors orthogonal to $\Lambda$; their subset is identified with the unit sphere in 
 $T_q\partial K$. 
 We will also treat $\mcr_\mcl$ as a map acting on the space of oriented lines passing through  
 the point $q\in\partial K$ where the normal vector $\vec n(q)$ to $\partial K$ is parallel to  $\La$: 
 the vectors $\vec n(A)$, $\vec n(B)$ are  identified with the lines through $q$ that are parallel and co-directed with them. We use the following 
 \begin{proposition} \label{proeq}  Fix an arbitrary point $q\in\partial K$. Let $\mcl$ be the class of lines parallel to the normal 
 line to $\partial K$ at $q$. Then the $T$-billiard reflection map $R_{K,q}$ coincides with the corresponding map $\mcr_\mcl$ considered as a map acting on  oriented lines passing through  $q$ and restricted to the subset of 
 lines oriented "out of $K$" at $q$.
 \end{proposition}
 The proposition follows from the discussion at the beginning of the section.
  
  \begin{theorem} \label{prodim2'}\footnote{Very recently a new, direct and simple proof of Theorems  \ref{prodim2'} and \ref{prodim2} was jointly obtained  by Vladimir Matveev with the author,  based on new ideas suggested by Matveev} Let $\alpha\subset\rr^n$ be a germ of hypersurface with 
  positive definite second fundamental form. 
   Consider the tautological projection $S^{n-1}\to\rp^{n-1}$, which is 
 the double covering sending each unit vector to its projectivization: the point representing its 
 ambient line. Let for every class $\mcl$ of parallel lines (that contains at least one line intersecting 
 $\alpha$ twice) the involution $\mcr_\mcl$ given by (\ref{jmcl}) (with $\partial T=\alpha$) 
 be the lifting of a projective involution $\rp^{n-1}\to\rp^{n-1}$. Then $\alpha$ is a quadric.
 \end{theorem}
   Theorem \ref{prodim2} follows from Theorem \ref{prodim2'} and Proposition \ref{proeq}.  
  
  \begin{remark} Maxim Arnold and Sergei Tabachnikov proved a result closely related to Theorem \ref{prodim2'}, see  \cite[theorem 3]{art}: {\it if a smooth strictly convex closed 
  hypersurface $S$ admits a normal vector field 
  $N$ such that for every $x,y\in S$ one has $<N(x),x-y>=-<N(y),x-y>$, then $S$ is an ellipsoid 
  (ellipse in two dimensions).}
  \end{remark} 
  
  \subsection{Applications to Finsler billiards}
  
Let us recall the following well-known material from Finsler geometry, 
see, e.g., \cite{abate, alv1, alv2, arn1, bao, gt}. A {\it Finsler manifold} is a manifold $M$ equipped with a {\it Finsler structure:}  a family of smooth strictly convex closed hypersurfaces $I_x\subset T_xM$ centrally-symmetric with respect to zero; they are called {\it indicatrices.}  The {\it Finsler norm} on $T_xM$ is the norm for which the indicatrix $I_x$ is the unit sphere. For every vector $v\in T_xM$ its Finsler norm is denoted by $L(x,v)$. A {\it Finsler geodesic} is a curve $\gamma:[t_0,t_1]\to M$ minimizing the functional 
$$|\gamma|:=\int_{t_0}^{t_1}L(\gamma(t),\dot\gamma(t))dt$$
among curves with fixed endpoints. 

\begin{example} Recall (see, e.g., \cite{gt}) that a Finsler structure on $\rr^n$ is called {\it Minkowski,} if it is translation invariant. A Finsler structure on a domain $U\subset\rr^n$ is {\it projective,} (or {\it projectively flat}), if all its geodesics are lines. It is well-known that Minkowski Finsler structures  
 are projectively flat, see \cite[p. 280]{gt}, and so are the Riemannian metrics of  spaces forms of constant curvature in appropriate 
 chart. Conversely, a simply connected domain in $\rr^n$ equipped with a projectively flat Riemannian metric 
 is always isometric (up to rescaling the metric by a constant factor) to a domain in a space form of constant curvature (Beltrami Theorem, see, e.g.,  
    \cite[section VI.2]{schouten}).
 \end{example}

E.Gutkin and S.Tabachnikov introduced the following {\it Finsler reflection law} in their joint paper \cite{gt}. 
Let $S\subset M$ be a hypersurface (for simplicity, strictly convex). Let $x\in S$,  and let $a,b\in M$ be points 
such that there exist a geodesic arc from $a$ to $x$ and a geodesic arc from $x$ to $b$. Let $u$, $v$ denote their directing unit tangent vectors at $x$. We suppose that $u$ and $v$ lie on different sides from the tangent hyperplane $T_xS$. We say that the {\it Finsler reflection with respect to the hyperplane $T_xS$ 
sends $u$ to $v$,} if $x$ is a critical point of the function $g(x):=d(a,x)+d(x,b)$: the sum of lengths of the above 
geodesic arcs with fixed endpoints $a$ and $b$. The following results were proved in \cite[pp. 281, 282]{gt}: 

1) Thus defined reflection 
transformation depends only on $x$ and $T_xS$ and is independent on $a$, $b$, $S$.

2) The reflection is given by a well-defined diffeomorphic involution $I_x\to I_x$ of the indicatrix that  
can be defined geometrically in one of the following equivalent ways.

a)  Consider the {\it Legendre transform} $D:I_x\to 
T_x^*M\setminus\{0\}$, which sends each vector $v\in I_x$ to the covector $w\in T_x^*M$ for which 
$w(v)=1$ and the kernel $\ker w\subset T_xM$ is parallel to the tangent hyperplane $T_vI_x\subset T_xM$. 
The latter pair of conditions is equivalent to the pair of conditions  $w(v)=1$ and  $|w(u)|\leq1$ for every $u\in I_x$. The hypersurface $J_x:=D(I_x)\subset T_x^*M$ is a closed convex centrally-symmetric hypersurface, which is called the {\it figuratrix} or the {\it dual to the indicatrix $I_x$}. The equivalent Finsler reflection law states that $u\in I_x$ {\it is reflected to} 
$v\in I_x$ with respect to a codimension one subspace $H=T_xS\subset T_xM$, if the linear functional 
$D(u)-D(v)\in T_x^*M$ vanishes on $H$. Recall that the {\it Legendre transformation is involutive:} the inverse 
$D^{-1}:J_x\to I_x$ coincides with the Legendre transformation defined by the hypersurface $J_x$ instead of 
$I_x$. 

b) The vector $u\in I_x$ is reflected to $v\in I_x$, if the hyperplanes tangent to $I_x$ at $u$ and at $v$ 
respectively and the reflecting hyperplane $H$ are concurrent: that is, either they are parallel, or they intersect themselves by a codimension two affine subspace in $T_xM$. See Figure 3.  

\begin{figure}[ht]
  \begin{center}
   \epsfig{file=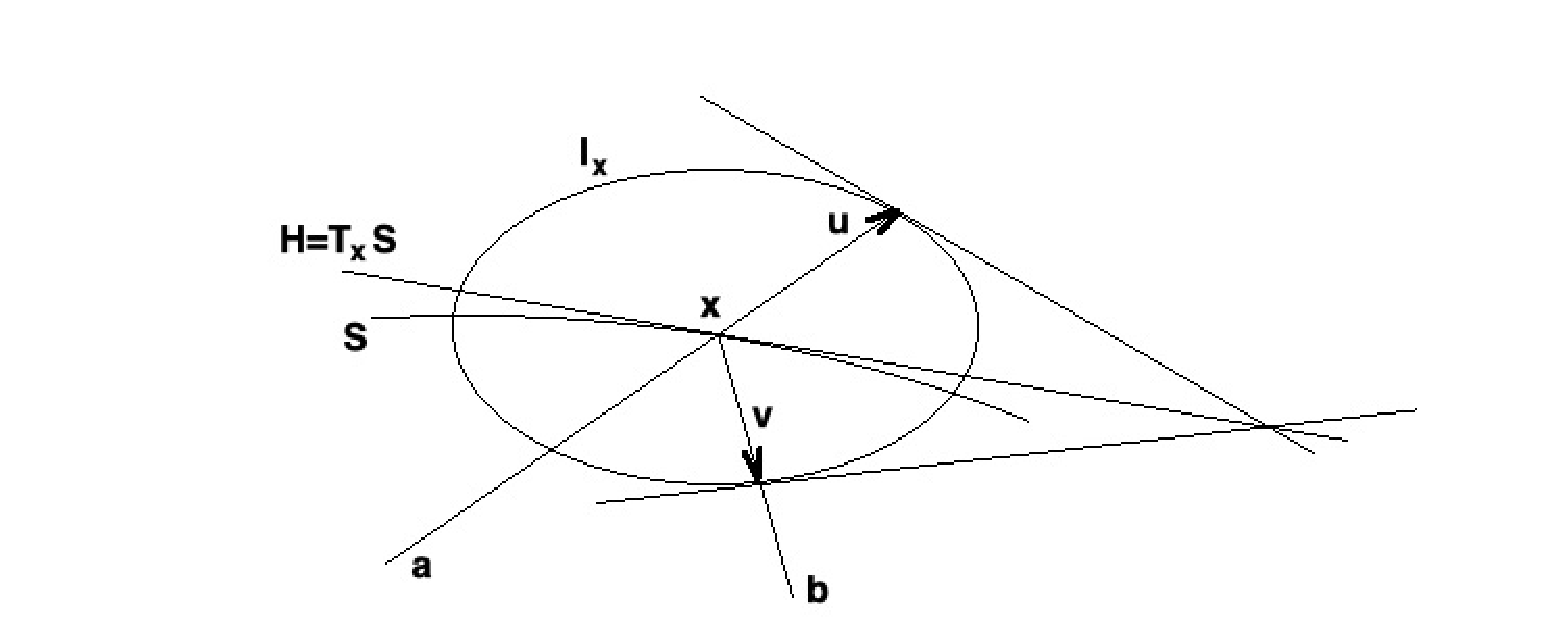, width=30em}
    \caption{Finsler reflection law.}
    \label{fig:1}
  \end{center}
  \end{figure}

Let us now consider that our Finsler manifold $M$ is a domain in $\rr^n$ equipped with some Finsler structure. 
For every $x\in M$ we identify the cotangent space $T_x^*M$ with the corresponding tangent space $T_xM$ via the standard Euclidean scalar product. Then the figuratrix $J_x$ is identified with a centrally-symmetric strictly 
convex closed hypersurface in $T_xM$, and the Legendre transform $D$ sends each vector $u\in I_x$ to 
the vector $D(u)\in T_xM$ orthogonal to $T_uI_x$ and such that $<D(u),u>=1$. In other terms, {\it the Legendre transform} $D:I_x\to J_x$ is {\it the Gauss map} multiplied by a positive function. 

Given a strictly convex closed surface $S\subset M$, we can define the {\it Finsler billiard} on $S$ acting in the space of oriented Finsler geodesics intersecting $S$. Namely, take an oriented geodesic $\ell$ intersecting $S$ 
and its last intersection point $q$ with $S$ in the sense of orientation. Its reflection image is the geodesic 
$\ell^*$ issued from $q$, directed inside the convex domain bounded by $S$. 

The next proposition is well-known and follows immediately from definition, involutivity and the above Gauss map interpretation of the Legendre transform $D:I_x\to J_x$. 
  
  \begin{proposition} \label{profinsler} (see, e.g., \cite[p. 2011]{ako1}). 
  Let $T\subset\rr^n$ be a strictly convex bounded centrally symmetric convex body with smooth boundary. 
  Then for every strictly convex bounded body $K\subset\rr^n$ with smooth boundary 
  the $T$-billiard in $K$ (with reflection law given in 
  Definition \ref{deftbil}) is the Finsler billiard  in the Minkowski Finsler space, where the indicatrix $I_0$ is  dual to the boundary $\partial T$. Namely, 
  $\partial T=J_0=D(I_0)$ (after translation of the body $T$ so that its center of symmetry becomes the origin). 
  \end{proposition}
  
  Theorem \ref{prodim2} together with Proposition \ref{profinsler} imply the following corollaries.
  
  \begin{corollary} \label{cor1} Let in a Minkowski Finsler space $\rr^n$ the Finsler reflection law be projective. 
  Then the Finsler unit sphere is an ellipsoid. In particular, there exists an affine transformation $\rr^n\to\rr^n$ 
  sending the Minkowski Finsler structure to the standard Euclidean metric and thus, 
  sending orbits of each Finsler billiard  to orbits of an Euclidean billiard with the standard reflection law. 
  \end{corollary}
  \begin{proof} Let $J=J_0=D(I_0)$ denote the figuratrix: the closed hypersurface dual to the unit sphere of the Finsler structure. Let $T$ denote the convex body bounded by $J$. For every bounded strictly convex body  $K\subset\rr^n$ the $(K,T)$-billiard is projected to 
  the $T$-billiard. The $T$-billiard reflection law is the Minkowski Finsler reflection law (Proposition 
  \ref{profinsler}). If it is projective, then $\partial T$ is an ellipsoid (Theorem \ref{prodim2}): 
  $\partial T=E:=\{<Ax,x>=1\}$, where $A$ is a symmetric positive definite matrix. But the dual to an ellipsoid $E$ 
    is the ellipsoid $E^*:=\{<A^{-1}x,x>=1\}$. Indeed, for every $x\in E$ the vector $Ax$ is normal to $T_xE$, 
    being half-gradient of the quadratic form $<Ax,x>$. Its scalar product with $x$ is equal to 
    $<Ax,x>=1$, since $x\in E$. Hence, $Ax=D(x)$. On the other hand, $<A^{-1}(Ax),Ax>=<x,Ax>=1$, since the matrix $A$ is symmetric. Therefore, $Ax\in E^*$. Finally, the hypersurface $D(\partial T)$ is an ellipsoid. It 
    coincides with the indicatrix $I_0$, by involutivity of the Legendre transform. 
    Corollary \ref{cor1} is proved.
    \end{proof} 
  
  \begin{corollary} \label{cor2} Let $U\subset\rr^n$ be a simply connected domain equipped with a Finsler structure. 
  Let the corresponding Finsler reflection law be projective. Then the Finsler structure is Riemannian, i.e.,  
  the corresponding unit spheres are ellipsoids.
  \end{corollary}
  
  \begin{proof} At each point $q\in U$ for every codimension one vector subspace $H\subset T_q\rr^n$ 
   the involution defining the Finsler reflection law with respect to $H$ is the same, as in the Minkowski Finsler 
   space $\rr^n$ with the indicatrix $I_0$ being translated image of the indicatrix $I_q$. Therefore, if 
   the reflection law in $T_q\rr^n$ is projective, then this is also the case in the  
   Minkowski space. Hence, the corresponding indicatrix $I_0$ is an ellipsoid (Corollary \ref{cor1}). 
   This proves  Corollary \ref{cor2}. 
   \end{proof}

    \begin{corollary} \label{cor3} Let $U\subset\rr^n$ be a simply connected domain equipped with a projectively flat Finsler structure. Let the corresponding Finsler reflection law be projective. Then (up to rescaling the metric by constant factor) $U$ is isometric to 
    a domain in a $n$-dimensional space form of constant curvature: either Euclidean space $\rr^n$, or the unit  
    sphere $S^n$, or the hyperbolic space $\hh^n$.
    \end{corollary}
    \begin{proof} The Finsler structure is Riemannian (Corollary \ref{cor2}). The simply connected Riemannian manifold $U$ is projectively flat, by assumption. Therefore, it is isometric (up to rescaling the metric by constant 
    factor) to a domain in a space form of constant curvature (Beltrami Theorem, see  
    \cite[section VI.2]{schouten}). Corollary \ref{cor3} is proved.
    \end{proof}

\section{Proof of Theorem \ref{prodim2'} in two dimensions}  

In this section  we consider that $n=2$; thus, $\alpha$ is a planar curve.

Fix a point $p\in\alpha$. Let $\Gamma=\Gamma_p$ be the osculating conic for the curve $\alpha$ at $p$. This is the unique conic tangent to $\alpha$ with contact of order greater than 4, that is, having the same 4-jet, as $\alpha$ at $p$. 

Recall that a point $p$ of a curve $\alpha$ is called {\it sextactic,} if its osculating conic 
at $p$ has contact of order greater than five with $\alpha$, see \cite[p.12]{buchin}. 

The proof of Theorem \ref{prodim2'} is based on the following well-known theorem. 

\begin{theorem} \label{tconic} If all the points of a connected planar curve $\alpha$ are sextactic, then $\alpha$ is a conic (or a connected component of a hyperbola). 
\end{theorem}

\begin{proof} The sextactic points are exactly those points where the affine curvature 
of the curve has zero derivative  \cite[p. 12, formula (3.33)]{buchin}. Therefore, $\alpha$ has constant 
affine curvature, and hence,  is a conic, see \cite[pp. 9-10]{buchin}. 
\end{proof} 

Therefore, for the proof of Theorem \ref{prodim2'} for $n=2$ it suffices to prove the following lemma. 

\begin{lemma} \label{losc} Fix an arbitrary 
point $p\in\alpha$. Let $L_p$ denote the line tangent to 
$\alpha$ at $p$. 
Let $\mcr_\mcl$ be the involution  (\ref{jmcl}) defined by 
 $\partial T=\alpha$ and the class $\mcl$ of lines parallel to $L_p$. Let 
 $\mcr_{\mcl}$ be the lifting of a projective involution of $\rp^1$. Then  $p$ is a sextactic point of the curve 
 $\alpha$. 
\end{lemma}

\begin{proof} We prove the lemma by contradiction. Suppose the contrary: at some point $p\in\alpha$ 
the conic $\Gamma=\Gamma_p$ and the curve $\alpha$ are locally graphs of functions that 
differ at 5-th Taylor terms. For both of them consider the corresponding  involutions $\mcr_\mcl$, 
$$f:=\mcr_{\alpha,\mcl}, \  \ g:=\mcr_{\Gamma,\mcl},$$
defined by the class $\mcl$ of lines parallel to $L_p$ and by the curves $\alpha$ and $\Gamma$ 
respectively. These are germs of one-dimensional $C^5$-smooth involutions at their common 
fixed point $O$ corresponding to the normal vector $\vec n(p)$.  We prove the following proposition, which is the main part of the proof of 
the lemma.
\begin{proposition} \label{jets=} In the above assumptions the involutions $f$ and $g$ have the same 2-jet at $O$ 
(and even the same 3-jet), but $f\not\equiv g$. 
\end{proposition}

The involutions $f$ and $g$  are both liftings of projective involutions. This follows from assumptions for the involution $f$, which is defined by the curve $\alpha$, and by Corollary \ref{cquad} for 
the involution $g$, which is defined by the conic $\Gamma$.  This together with the next 
proposition on a general pair of projective involutions will bring us to contradiction with Proposition 
\ref{jets=}. 

\begin{proposition} \label{jetneq} Two arbitrary projective involutions $\rp^1\to\rp^1$ with a common fixed point 
either coincide, or have different 2-jets at the fixed point.
\end{proposition}

\begin{proof} Let  $f$, $g$ denote the involutions in question, and let $O$ denote their common fixed point.  
The map $f$ has yet another fixed point $A$. 
In an affine chart $\rr_{t}$ on $\rp^1\setminus\{ A\}$ centered at $O$  
one has 
$$f(t)=-t, \ \ g(t)=-\frac t{1+ct}=-t+ct^2+O(t^3), \ \text{ as } x\to0.$$
One has $c\neq0$, if and only if $f\neq g$. In this case the germs $f$ and $g$ obviously have 
different 2-jets. The proposition is proved.
\end{proof}

\begin{proof} {\bf of Proposition \ref{jets=}.} The involutions $f$ and $g$ are defined to permute the 
exterior unit normal vectors to $\alpha$ and $\Gamma$ respectively at points of intersection of 
a line $\ell$ parallel to $L_{p}$ and the curve in question. The exterior normal vector $\vec n(y)$ 
to the underlying curve at a point $y$ will be identified with the oriented tangent line $L_y$ to the curve at $y$. 
Its orienting vector  is  obtained from the normal vector $\vec n(y)$ by  rotation by $\frac\pi2$. Both 
$\vec n(y)$ and $L_y$ will be identified with the 
slope of the line $L_y$: the tangent of its azimuth with respect to the line $L_p$. 
Thus, we can and will consider that the involutions $f$ and $g$ act on the slopes of the oriented  
tangent lines to the underlying curves. 

Let us take positively oriented orthogonal affine coordinates $(x,y)$ centered at $O=p$ so that the $x$-axis coincides with the oriented tangent line $L_p=L_O$. Then the curves $\alpha$ and $\Gamma$ both lie in the upper half-plane (by choice of orientation). We can and will consider them as  graphs of functions:
$$\alpha=\{ y=h_\alpha(x)\}, \ \ \Gamma=\{ y=h_\Gamma(x)\}, \ h_\alpha(x)\simeq h_\Gamma(x)
\simeq\frac12x^2, \text{ as } x\to0.$$
Indeed, the latter asymptotics a priori holds with $\frac12x^2$ replaced by $\nu x^2$ with $\nu>0$. 
One can achieve that $\nu=\frac12$ by applying a homothety, which does not change 
our projectivity and jet (in)equality assumptions. One has 
\begin{equation} h_\alpha(x)=h_{\Gamma}(x)+cx^5+o(x^5), \ \ c\neq0,\label{hc5}\end{equation}
since the 5-jets of the curves $\alpha$ and $\Gamma$ are supposed to be distinct, while their  
4-jets are equal. See Figure 4. 
\begin{figure}[ht]
  \begin{center}
   \epsfig{file=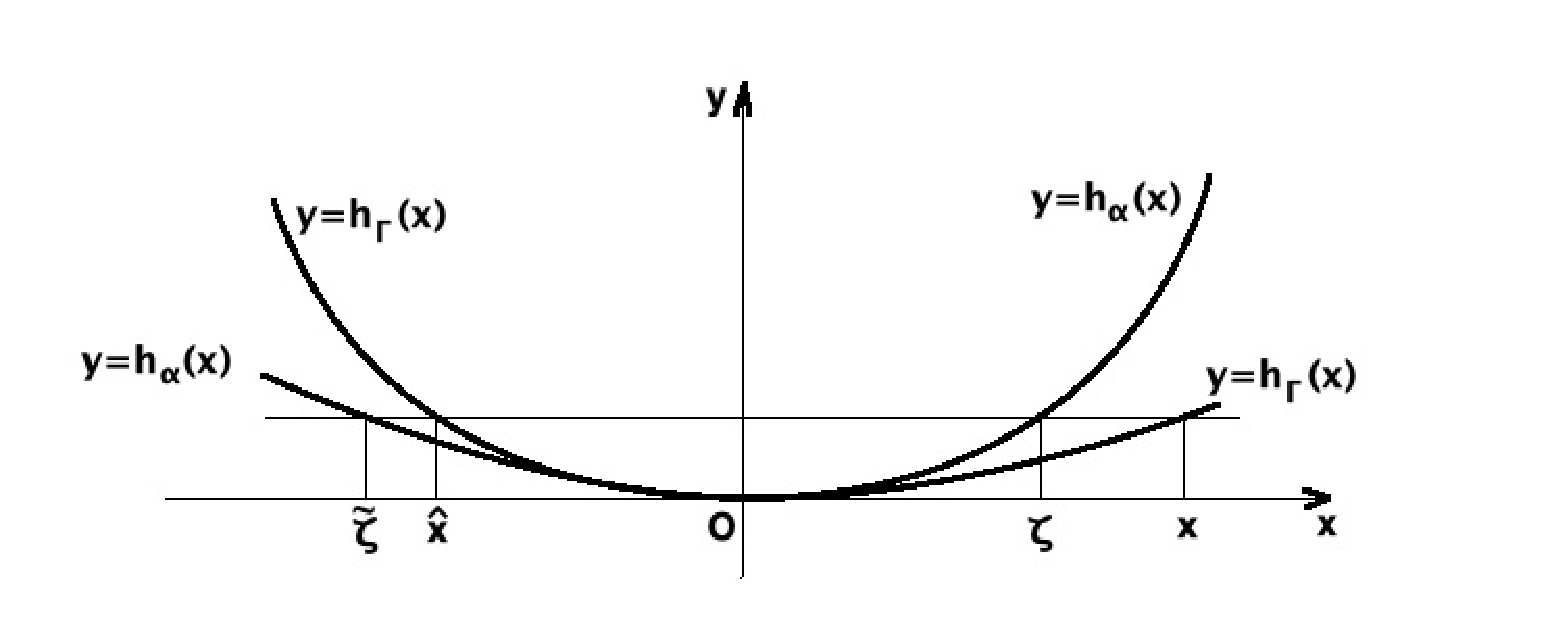, width=30em}
    \caption{The curve $\alpha$ and the conic $\Gamma$ as graphs of functions near $p=O$ that 
    differ in 5-th Taylor term.}
    \label{fig:1}
  \end{center}
\end{figure}
For every $x$ the slope of the curve $\beta=\alpha, \Gamma$ at the point $(x,h_\beta(x))$ 
is equal to $h'_\beta(x)$, and one has 
\begin{equation}h'_\beta(x)\simeq x, \ \ h'_\alpha(x)-h'_\Gamma(x)=5cx^4+o(x^4),
\label{h'b}\end{equation}
by (\ref{hc5}). For every $x$ small enough let 
\begin{equation} \zeta=\zeta(x) \ \text{ denote the point for which } \ h_\alpha(\zeta)=h_\Gamma(x) \text{ and } 
\zeta\simeq x,\label{defz}\end{equation}
\begin{equation}\hat x=\hat x(x) \ \text{ denote the  point distinct from } x \text{ where } \ h_\Gamma(\hat x)=h_\Gamma(x),\label{defxz}\end{equation}
\begin{equation}\tilde\zeta=\tilde\zeta(x) \ \text{ denote the  point distinct from } \zeta \text{ where } \ 
h_\alpha(\tilde\zeta)=h_\alpha(\zeta),\label{deftz}\end{equation}
see Figure 4. 

\medskip

{\bf Claim 1.} {\it One has}
\begin{equation}\zeta\simeq x, \  \hat x\simeq-x, \ \tilde\zeta\simeq-\zeta, \  h_\Gamma'(x)\simeq x, \ h_\alpha'(x)\simeq x,\label{hgha}\end{equation}
\begin{equation}
\zeta=x-cx^4+O(x^5), \ \  \tilde\zeta=\hat x-cx^4+O(x^5),\label{xzeta}\end{equation}
   \begin{equation}h_\alpha'(\zeta)=h_\Gamma'(x)+4cx^4(1+o(1)), \ \ h_{\alpha}'(\tilde\zeta)=
h_\Gamma'(\hat x)+4cx^4(1+o(1)).\label{h'}\end{equation}

\begin{proof} Formulas (\ref{hgha}) follow from definition, since 
$h_\beta(x)\simeq\frac12x^2$ for $\beta=\alpha,\Gamma$. One has  
$h_\Gamma(x)=h_\alpha(\zeta)=h_\Gamma(\zeta)+c\zeta^5+o(\zeta^5)$, $\zeta\simeq x$. Hence, 
$$h_\Gamma(\zeta)-h_\Gamma(x)=-c\zeta^5(1+o(1)).$$
The difference in the latter left-hand side is asymptotic to 
$h'_\Gamma(x)(\zeta-x)\simeq\zeta(\zeta-x)$. This implies the first asymptotic equality in 
(\ref{xzeta}). Applying the same argument to $\hat x$ and $\wt \zeta$ yields the second equality in 
(\ref{xzeta}). Let us prove the first equality in (\ref{h'}); the proof of  the second one is analogous. 
One has 
$$h_\alpha'(\zeta)-h_\Gamma'(x)=(h_\alpha'(\zeta)-h_\Gamma'(\zeta))+(h_\Gamma'(\zeta)-h_\Gamma'(x))$$
$$=5c\zeta^4+o(\zeta^4)+h_\Gamma''(x)(\zeta-x)+O((\zeta-x)^2).$$
Substituting the asymptotics $\zeta\simeq x$, $h_\Gamma''(x)\simeq 1$, $\zeta-x\simeq-cx^4$,  to the latter right-hand side  yields  (\ref{h'}) and finishes the proof of the claim.
\end{proof} 

Recall that the involutions $f$ and $g$ act on the slopes of tangent lines so that 
\begin{equation}g(h_\Gamma'(x))=h_\Gamma'(\hat x), \ f(h_\alpha'(\zeta))=h_\alpha'(\wt\zeta),
\label{gfh}\end{equation}
\begin{equation}h_\Gamma'(x)\simeq h_\alpha'(x)\simeq h_\alpha'(\zeta)\simeq x,\label{simx}
\end{equation}
\begin{equation}h_\Gamma'(\hat x)\simeq h_\alpha'(\hat x)\simeq h_\alpha'(\tilde\zeta)\simeq
\hat x\simeq -x.\label{sim-x}\end{equation}

\medskip

{\bf Claim 2.} {\it One has} 
\begin{equation} f(t)=g(t)+ 8ct^4(1+o(1)).\label{ftgt}\end{equation}

\begin{proof} Clearly one has $f(t)\simeq g(t)\simeq -t$, $f'(t)\simeq g'(t)\simeq -1$. This together 
with (\ref{gfh})-(\ref{sim-x}) and (\ref{h'}) implies that 
$$f(h_\Gamma'(x))=f(h_\alpha'(\zeta))-f'(h_\Gamma'(x))(h_\alpha'(\zeta)-h_\Gamma'(x))+
O((h_\alpha'(\zeta)-h_\Gamma'(x))^2)$$
$$=h_\alpha'(\wt\zeta)+4cx^4(1+o(1))=h_\Gamma'(\hat x)+(h_\alpha'(\wt\zeta)-h_\Gamma'(\hat x))
+4cx^4(1+o(1))$$
$$=g(h_\Gamma'(x))+8cx^4(1+o(1)).$$
This together with the asymptotics $h_\Gamma'(x)\simeq x$ proves (\ref{ftgt}).
\end{proof}

Claim 2 immediately implies the statements of Proposition \ref{jets=}.
\end{proof}

Thus, $f$ and $g$ are distinct projective involutions  that have 
the same 2-jet at a common fixed point. This contradicts Proposition \ref{jetneq}. The contradiction 
thus obtained finishes the proof of Lemma \ref{losc}.
\end{proof}

Lemma \ref{losc} together with Theorem \ref{tconic} imply that the curve $\alpha$ is a conic. 
This proves Theorem \ref{prodim2'}.

\section{Proof of Theorem \ref{prodim2'} in higher dimensions}
\subsection{Plan of proof of Theorem \ref{prodim2'}}
In Subsection 3.2 we prove that each planar section of the hypersurface $\alpha$ 
is a conic. This together with \cite[theorem 4]{art} will imply that  $\alpha$ is a quadric: see Subsection 3.3.

We take an arbitrary plane $\Pi$ intersecting $\alpha$ transversally. Let $\beta$ denote their intersection curve. 
Our Main Lemma \ref{conjmany2} states that every point $O\in\beta$ is sextactic. 
This together with Theorem \ref{tconic} implies that $\beta$ is a conic. 

Let $\mcl$ denote the class of lines parallel to $T_O\beta$. Consider the corresponding germ of involution $f=\mcr_\mcl=\mcr_{\alpha,\mcl}$ acting on a domain in the unit sphere 
$S$, see (\ref{jmcl}). By assumption, it fixes the normal vector $\vec n(O)$ to $\alpha$ at $O$, and its 
germ at $\vec n(O)$ is the lifting to $S$ of a projective involution. 
 Let $\gamma\subset\Pi$ denote the osculating conic of the curve $\beta$ at $O$. We have to show that they 
 have contact of order at least six. To do this, in Proposition \ref{piii} we construct an auxiliary 
 quadric $\Gamma$ containing $\gamma$, which we call {\it the osculating quadric to $\alpha$ along the 
 curve $\beta$,}  
 that is tangent to $\alpha$ at $O$ with contact of order at least three and for which  the normal vector fields 
 $\vec n_{\gamma}$, $\vec n_{\beta}$ 
 of the hypersurfaces $\Gamma$ and $\alpha$ along the curves $\gamma$ and $\beta$ respectively 
 differ by a quantity of at least cubic order. We introduce orthogonal coordinates 
 $(x_1,x_n)$ on $\Pi$ centered at $O$, with the $x_1$-axis being the common tangent line 
 to $\beta$ and $\gamma$ at $O$.  We deal with the above  normal  fields 
 as germs of parametrized curves $\vec n_{\gamma}(x_1)$, 
 $\vec n_\beta(x_1)$ in the unit sphere $S$. We identify $S$ with $\rp^{n-1}$  near 
 $\vec n(O):=\vec n_\gamma(O)=\vec n_\beta(O)$ via the tautological projection and denote a point 
 in $S$ and its projection to $\rp^{n-1}$ by the same symbol.  The involution $g=\mcr_{\Gamma,\mcl}$ is 
 a projective transformation, as is $f$, since $\Gamma$ is a quadric, see Corollary \ref{cquad}. 
 
 Step 1. Using the above asymptotics we will show 
 (Proposition \ref{pro=}) that the projective involutions coincide: $f=g$. To this end, for every $x_1$ 
 we consider the corresponding parameter values  
 $\zeta$, $\hat x_1$, $\tilde\zeta$ defined for the osculating planar curves $\gamma$ and $\beta$, 
 as in (\ref{defz})--(\ref{deftz}), see Figure 4. One  has 
 \begin{equation}f(\vec n_{\beta}(\zeta))=\vec n_{\beta}(\wt\zeta), \ \ g(\vec n_{\gamma}(x_1))=\vec n_{\gamma}(\hat x_1),\label{f=n=g}\end{equation}
   by the definition of the involutions $f$ and $g$. On one hand, using the above asymptotic properties of 
   the quadric $\Gamma$, we show that 
   $$\vec n_\gamma(x_1)-\vec n_\beta(\zeta)=O(x_1^3), \ \ 
   \vec n_\gamma(\hat x_1)-\vec n_\beta(\tilde\zeta)=O(x_1^3).$$
  This together with (\ref{f=n=g}) will imply that 
 $$g(\vec n_\gamma(x_1))-f(\vec n_\gamma(x_1))=O(x_1^3).$$
 On the other hand, we show in Claim 4 that if $f\neq g$, then the latter asymptotic equality cannot hold. 
 This will prove that $f=g$.

 Step 2. We prove Lemma \ref{conjmany2} by contradiction. 
 Suppose the contrary: the curves 
 $\beta$ and $\gamma$ have contact of order exactly five, i.e.,  they are graphs 
 $$\beta=\{ x_n=\nu_\alpha(x_1)\}, \ \gamma=\{ x_n=\nu_\Gamma(x_1)\}, \ \nu_\alpha(x_1)-
 \nu_\Gamma(x_1)\simeq cx_1^5, \ c\neq0.$$
 We calculate  asymptotics of difference of the  images in (\ref{f=n=g}) in two different ways and show that the  results contradict each other. 
 In more details, let $H$ denote the hyperplane orthogonal to  $T_O\beta$: it contains $\vec n(O)$. 
The projective involution $f$ fixes each point in $H$, and it has yet another fixed point denoted by 
$P\in \rp^{n-1}\setminus H$. We  construct a special affine chart $(z,y_1,\dots,y_{n-2})$ on  
 $\rp^{n-1}$ centered at  $\vec n(O)$ in which  $H=\{z=0\}$ and the projection $(z,y)\mapsto y$ is the 
 projection from the point $P$, thus $f$ fixes each $z$-fiber, so that 
$$z(\vec n_{\gamma}(x_1))\simeq z(\vec n_{\beta}(\zeta))\simeq x_1, \ \ \text{ as } \ x_1\to0.$$
We show that  

  a) the projective involution $f$  acts on each $z$-fiber  as $z\mapsto-z+O(z^2)$, where the $O(z^2)$ 
depends analytically on  $y$; 

b) the points $\vec n_\beta(\zeta)$ and $\vec n_\gamma(x_1)$  lie in $z$-fibers with $O(x_1^3)$-close  projections to the $y$-coordinates; 

c) their  $z$-coordinates  differ by a quantity $4cx_1^4(1+o(1))$; 

d) analogously 
one has  $z(\vec n_{\beta}(\wt\zeta))\simeq z(\vec n_{\gamma}(\hat x_1))\simeq -x_1$ and 
\begin{equation} z(\vec n_{\beta}(\wt\zeta))-z(\vec n_{\gamma}(\hat x_1))=4cx_1^4(1+o(1)).
\label{whwh}\end{equation} 
Statements a)--c) together imply that 
 the $f$-images of the points $\vec n_{\beta}(\zeta)$, $\vec n_{\gamma}(x_1)$  have $z$-coordinates that differ by $-4cx_1^4(1+o(1))$ plus a quantity $O(x_1^3)\times O(x_1^2)=O(x_1^5)$. 
Thus, 
\begin{equation}z(f(\vec n_\beta(\zeta)))-z(f(\vec n_\gamma(x_1)))=-4cx_1^4(1+o(1)).\label{zgv}
\end{equation}

But (\ref{zgv}) contradicts to (\ref{whwh}), due to equality (\ref{f=n=g}) and since $f=g$ (Step 1).   
The contradiction thus obtained will prove Lemma \ref{conjmany2}.

\subsection{Proof that planar sections are conics}
Here we prove the following theorem.
\begin{theorem} \label{conjmany} Consider a $C^6$-smooth germ $\alpha$ of hypersurface in $\rr^n$  with positive definite second fundamental quadratic form. Fix a two-dimensional plane $\Pi$ intersecting $\alpha$ transversally. Set $\beta=\Pi\cap\alpha$. Let $O\in\beta$,  $L\subset\Pi$ be the line tangent to 
$\beta$ at $O$, and let  $\mcl$ denote the class of lines parallel to $L$. 
Consider the corresponding 
germ of involution $f=\mcr_{\mcl}=\mcr_{\alpha,\mcl}$ acting on a domain in the unit sphere $S=S^{n-1}$, see (\ref{jmcl}): its germ at its fixed vector $\vec n(O)=\vec n_{\alpha}(O)=\vec n_{\Gamma}(O)$. Let for every line $L$ tangent to $\beta$ the corresponding mapping $f$ be 
the litfing to $S$ of a projective transformation. Then the intersection curve $\beta$ is a conic. 
\end{theorem}

\begin{corollary} In the conditions of Theorem \ref{prodim2'} each planar section of the hypersurface 
$\alpha$ is a conic.
\end{corollary}

For the proof of Theorem \ref{conjmany} it suffices to prove the following lemma. 

\begin{lemma}\label{conjmany2} In the  conditions of Theorem \ref{conjmany} let for a given point $O\in\beta$ the corresponding map $f$ 
be the litfing to $S$ of a projective transformation.  Then $O$ is a sextactic point of the planar curve $\beta\subset\Pi$. 
 \end{lemma}
 
Lemma \ref{conjmany2}  together with Theorem \ref{tconic}  imply Theorem \ref{conjmany}. 
 
  \begin{proposition} \label{piii} Given a point $O\in\rr^n$, a germ $\alpha$ of $C^6$-smooth 
 hypersurface in $\rr^n$ at  $O$ 
 and a two-dimensional plane $\Pi$ through $O$ transversal to $\alpha$ 
 at $O$. Let $L$ denote the intersection of the plane $\Pi$ with the hyperplane tangent to $\alpha$ 
 at $O$. Set $\beta:=\Pi\cap\alpha$. Let the planar curve $\beta$ have positive curvature at $O$ (which is always the case if 
 $\alpha$ has positive second fundamental form at $O$).  Let $\gamma\subset\Pi$ denote the osculating conic at $O$ for the curve $\beta$. There exists a unique quadric $\Gamma\subset\rr^n$ containing the osculating conic $\gamma$ that satisfies the following statements:
 
 (i) The quadric $\Gamma$ is tangent to $\alpha$ at $O$ with contact of at least cubic order, 
 i.e., has the same tangent hyperplane and the same second fundamental quadratic form at $O$, as $\alpha$;
 
 (ii) Consider the osculating curves $\beta$ and $\gamma$ parametrized by the same parameter $t$, 
 $t(O)=0$,  so that $\beta(t)-\gamma(t)=O(t^4)$. The smooth unit normal vector fields on $\alpha$ and 
 $\Gamma$ with the same orientation at $O$, being restricted to the curves $\beta$ and $\gamma$, have the same 2-jet at $0$  as vector functions of $t$. 
 
 The quadric $\Gamma$ is called the {\bf osculating quadric along the curve $\beta$}  for the hypersurface $\alpha$.  It is unique in the class of complex quadrics.
 \end{proposition} 
 
 {\bf Addendum to Proposition \ref{piii}.} {\it The above quadric $\Gamma$ exists and is unique under a weaker assumption, with  $\gamma$ being just some conic  tangent to $\beta$ with 
 contact of order at least four (not necessarily at least five).}
 \medskip

 \begin{proof} We prove the addendum (a stronger statement). Thus, everywhere below we will assume 
that $\gamma$ is a conic tangent to $\beta$ with contact of order at least 4. 
 Condition (ii) is parametrization-independent: replacing $\gamma(t)$ by $\gamma(\tau(t))=
 \gamma(t)+O(t^4)$  changes the unit normal vector to $\Gamma$ at $\gamma(t)$ by a quantity of order at least three. 
  The system of conditions (i) and (ii) is invariant under projective transformations. Indeed, invariance of condition (i) is obvious. Condition (ii) means that at points $\beta(t)$ and $\gamma(\tau)$ 
that are close to each other of order $O(t^4)$ the corresponding tangent hyperplanes to $\alpha$ and $\Gamma$ respectively  are $O(t^3)$-close to each other as elements of Grassmannian manifold (which is in our case the dual projective space). The latter condition is obviously projective-invariant. 
Applying projective transformation of $\rp^n\supset\rr^n$ we can and will consider that $O$ is the 
origin in the affine chart $\rr^n$, the tangent hyperplane to 
 $\alpha$ at $O$ is the coordinate hyperplane $\{ x_n=0\}$, the plane $\Pi$ is the $(x_1,x_n)$-coordinate plane, 
 and the conic $\gamma$ is the parabola $\{ x_n=x_1^2\}$.  Set
 $$\wh x:=(x_2,\dots,x_{n-1}).$$
 Then each quadric $\Gamma$ tangent to the 
 hyperplane $\{ x_n=0\}$ at $O$ and containing $\gamma$ is given by a quadratic equation of the type 
 \begin{equation} x_1^2+x_1\sum_{j=2}^{n-1}c_jx_j+<A\wh x,\wh x>-x_n(1+\sum_{j=2}^{n-1}d_jx_j)=0. \label{quadeq}\end{equation}
 Or equivalently, 
 \begin{equation}\Gamma=\{ h_\Gamma(x_1,\wh x)-x_n=0\}, \ \ h_{\Gamma}(x_1,\wh x)=\frac{x_1^2+x_1\sum_{j=2}^{n-1}c_jx_j+<A\wh x,\wh x>}{1+\sum_{j=2}^{n-1}d_jx_j}\label{hgxx}\end{equation}
$$=(x_1^2+x_1\sum_{j=2}^{n-1}c_jx_j+<A\wh x,\wh x>)(1-\sum_{j=2}^{n-1}d_jx_j)+(\text{terms of order}\geq4).$$
 Here $A$ is a symmetric $(n-2)\times(n-2)$-matrix. Each germ of $C^6$-smooth 
 hypersurface $\alpha$ tangent to $\Gamma$ 
 at $O$ and having the same second quadratic form at $O$, as $\Gamma$, is locally a graph 
\begin{equation}\alpha=\{ h_\alpha(x_1,\wh x)-x_n=0\},\label{graphsag}\end{equation}
 where $h_{\alpha}$ is a $C^6$-smooth function of the type 
\begin{equation} h_{\alpha}(x_1,\wh x)=x_1^2+x_1\sum_{j=2}^{n-1}c_jx_j+<A\wh x,\wh x>+ (\text{ terms of order } \geq3).\label{hoterms}\end{equation} 
The intersection $\beta=\alpha\cap\Pi$ is a curve tangent to $\gamma$ with contact of order 
at least 4, if and only if the restriction to the $x_1$-axis of the  higher order terms in (\ref{hoterms}) has trivial 3-jet. This means that $h_{\alpha}$ does not contain Taylor term $x_1^3$. Let $s_j$ denote the Taylor 
coefficients of the function $h_{\alpha}$ at $x_1^2x_j$, $j=2,\dots,n-1$. Finally, 
\begin{equation}h_{\alpha}(x_1,\wh x)-h_\Gamma(x_1,\wh x)=x_1^2\sum_{j=2}^{n-1}(s_j+d_j)x_j+
x_1O(||\wh x||^2)+O(||\wh x||^3)\label{diffh}\end{equation}
$$+(\text{terms of order}\geq4 \text{ in } (x_1,\wh x)).$$

\medskip

{\bf Claim 3.} 
{\it Condition (ii) holds if and only if $d_j=-s_j$.}

\begin{proof}  The normal fields to the hypersurfaces $\xi=\alpha,\Gamma$ directed "down" are equal to 
\begin{equation} \vec n_{\xi}=\frac{(\frac{\partial h_\xi}{\partial x_1},\dots,\frac{\partial h_\xi}{\partial x_{n-1}}, -1)}
{\sqrt{1+\sum_{j=1}^{n-1}(\frac{\partial h_\xi}{\partial x_1})^2}}.\label{vecxi}\end{equation}
We have to compare them along the intersections of the hypersurfaces with the $(x_1,x_n)$-plane, which means comparing the restrictions of the vector functions (\ref{vecxi}) with $\xi=\alpha,\gamma$  to the 
$x_1$-axis. 
 The only terms in (\ref{diffh}) that can contribute to linear or quadratic terms in $x_1$ in the partial derivatives 
 $\frac{\partial h_\alpha}{\partial x_j}$ are 
\begin{equation} x_1^2\sum_{j=2}^{n-1}(s_j+d_j)x_j,\label{exterms}\end{equation}
 since 
the contributions of the terms $x_1O(||\wh x||^2)+O(||\wh x||^3)$ to the above partial derivatives vanish 
identically along the $x_1$-axis, as does the coordinate vector function $\wh x$, and 
 terms of order at least 4 contribute terms of order at least three. The contributions of  terms (\ref{exterms}) 
 to the above partial derivatives are quadratic, and the derivatives in question are small of order at least one. 
 Therefore, the contributions of  terms (\ref{exterms})  to the denominators in (\ref{vecxi}) for $\xi=\alpha$ are of order at least three.  Finally, the statement saying that the vector functions 
 $\vec n_\alpha$, $\vec n_\Gamma$ have the same 2-jet along the $x_1$-axis is equivalent to the same 
 statement for the numerators (the gradient fields) in (\ref{vecxi}). But the restriction to the $x_1$-axis of the 
 difference of the numerators in (\ref{vecxi}) 
 correspoding to $\xi=\alpha,\Gamma$ is equal (up to terms of order at least three) to the contribution of (\ref{exterms}): that is,  to $(0,(s_2+d_2)x_1^2,\dots,(s_{n-1}+d_{n-1})x_1^2,0)$. The latter contribution 
 has trivial second jet, if and only if it vanishes, i.e., $d_j=-s_j$ for all $j=2,\dots,n-1$. The claim is proved.
 \end{proof}

Thus, all the terms in (\ref{quadeq}) are uniquely determined by conditions (i) and (ii). This proves Proposition 
\ref{piii} and its addendum.
\end{proof}
  Consider now new, affine orthogonal coordinates $(x_1,\dots,x_n)$  centered at $O$ such that the $x_1$-axis is the common tangent line $L$ to the curves $\gamma$ and $\beta$, and the $x_n$-axis lies in the plane $\Pi$. 
   In what follows we consider that the curves $\gamma,\beta\subset\Pi$ have unit curvature at $O$: 
  one can achieve this by rescaling the coordinates $(x_1,\dots,x_n)$ by homothety. 
  The hypersurfaces $\xi=\Gamma, \alpha$ are graphs $\{ x_n=h_\xi(x_1,\dots,x_{n-1}\}$ of 
  functions $h_{\Gamma,\alpha}$. Their intersections with $\Pi$, that is, the 
  curves $\gamma$ and $\beta$, are graphs $\{ x_n= \nu_{\xi}(x_1):=h_\xi(x_1,0,\dots,0)\}$. 
  For every $x_1$ let $\hat x_1$, $\zeta$, $\hat\zeta$ be the same, as in (\ref{defz})--(\ref{deftz}): 
  $$\zeta=\zeta(x_1) \ \text{ is the point for which } \ \nu_\alpha(\zeta)=\nu_\Gamma(x_1), \text{ and } 
\zeta\simeq x_1,$$
$$\hat x_1=\hat x_1(x_1) \ \text{ is the  point distinct from } x_1 \text{ where } \ \nu_\Gamma(\hat x_1)=\nu_\Gamma(x_1),$$
$$\tilde\zeta=\tilde\zeta(x_1) \ \text{ is the  point distinct from } \zeta \text{ where } \ 
\nu_\alpha(\tilde\zeta)=\nu_\alpha(\zeta),$$
see Figure 4 (with $h$ replaced by $\nu$). We consider that the curves 
  $\gamma$ and $\beta$ are parametrized by the parameter $x_1$ as graphs of functions 
  $\nu_{\Gamma}(x_1)$ and $\nu_{\alpha}(x_1)$ respectively. Recall that we assume that 
  the curves $\gamma$ and $\beta$ have tangency of contact of order exactly 5. Therefore,
   \begin{equation}\nu_\alpha(x_1)-\nu_\Gamma(x_1)=cx_1^5(1+o(1)), \ \text{ as } x_1\to0; \ \ c\neq0.\label{ct5}\end{equation}
   Recall that we denote $f=\mcr_{\alpha,\mcl}$ and $g=\mcr_{\Gamma,\mcl}$. Here $\mcl$ is the class of lines 
   parallel to the line $L$ tangent to the curves $\beta$, $\gamma$ at $O$.

  \begin{proposition}  \label{pasmul} 
  1) Consider the restrictions $\vec n_\gamma(x_1)$, $\vec n_\beta(x_1)$ 
  to the curves $\gamma$ and $\beta$ 
  of the exterior unit normal vector fields to $\Gamma$ and $\alpha$ respectively as functions of $x_1$.  One has 
  \begin{equation} g(\vec n_\gamma(x_1))=\vec n_\gamma(\hat x_1), \ f(\vec n_\beta(\zeta))=
  \vec n_\beta(\tilde\zeta),\label{gnfg}\end{equation}
  \begin{equation}\vec n_\gamma(x_1)-\vec n_\beta(\zeta)=O(x_1^3), \
   \vec n_\gamma(\hat x_1)-\vec n_\beta(\tilde\zeta)=O(x_1^3).\label{ot3}\end{equation}
   
  2) Consider the orthogonal projections to $\Pi$ of the vector fields $\vec n_{\gamma}$, 
  $\vec n_{\beta}$. Let $\psi_\gamma(x_1)$ and $\psi_\beta(x_1)$ denote the arguments 
  of the projections (their azimuths with respect to the $x_1$-axis). One has 
  \begin{equation} \psi_\beta(\zeta)-\psi_\gamma(x_1)=4cx_1^4(1+o(1)),\label{psizg}\end{equation}
  \begin{equation} \psi_\beta(\tilde\zeta)-\psi_\gamma(\hat x_1)=4cx_1^4(1+o(1)),\label{psizg'}
  \end{equation}
    \end{proposition}
  
  \begin{proof} Formulas (\ref{gnfg}) follow from definition. One has 
  \begin{equation}\vec n_\beta(x_1)-\vec n_\gamma(x_1)=O(x_1^3).\label{vec3}\end{equation}
  This follows from the fact that the  restrictions to $\gamma$, $\beta$ of the normal vector 
  fields 
  of the hypersurfaces $\Gamma$ and $\alpha$ have the same 2-jet at $O$, which in its turn 
  follows from Statement (ii) of Proposition \ref{piii}. One has $\zeta-x_1=O(x_1^4)$, by 
  (\ref{xzeta}). This together with (\ref{vec3}) implies the first formula in (\ref{ot3}), and the 
  proof of the second formula is analogous. Each argument $\psi_{\beta,\gamma}$ in (\ref{psizg}), (\ref{psizg'}) 
  is equal to $-\frac\pi2$ plus the azimuth of the  tangent line (oriented to the right) to the curve in question  at the corresponding point. This follows from the fact that the projection to $\Pi$ of the restriction to 
  $\beta$ ($\gamma$) of the normal vector field to 
  $\alpha$ ($\Gamma$) is orthogonal to $\beta$ ($\gamma$). 
  The tangents of the latter azimuths are the derivatives in  (\ref{h'}). 
Formulas (\ref{psizg}) and (\ref{psizg'}) follow from formulas (\ref{h'}), the two last statements and the fact that $\tan x$ has unit derivative at the origin. Proposition \ref{pasmul} is proved.
  \end{proof}
  
  \begin{proposition} \label{pro=} 
   The  involutions $f=\mcr_{\alpha,\mcl}$ and $g=\mcr_{\Gamma,\mcl}$ coincide.
  \end{proposition}
  
  \begin{proof} One has 
    \begin{equation} g(\vec n_{\gamma}(x_1))-f(\vec n_{\gamma}(x_1))=O(x_1^3), \ \text{ as } x_1\to0.\label{o3x}
  \end{equation}
This follows from the fact that the  points $\vec n_\gamma(x_1)$, $\vec n_\beta(\zeta)$ 
are $O(x_1^3)$-close, as are their $g$- and $f$- images respectively, see (\ref{ot3}) and 
(\ref{gnfg}), 
 and smoothness of the maps $f$ and $g$. On the other hand, we prove the following claim. 

{\bf Claim 4.} {\it If $f\neq g$, then  $g(\vec n_{\gamma}(x_1))-f(\vec n_{\gamma}(x_1))$ is not $O(x_1^3)$, as $x_1\to0$.} 

\begin{proof} The projective involutions $f$ and $g$  have common hyperplane $H$ of fixed points, which 
  is the projectivization of the codimension one vector subspace $\{ x_1=0\}$ orthogonal to $L$. 
  Each involution $v=f,g$ has yet another fixed point $P_v\notin H$, fixes each line through $P_v$ and 
  acts there as a projective involution fixing $P_v$ and its intersection point with $H$. 
  Inequality $f\neq g$ is equivalent to the inequality $P_f\neq P_g$. 
  
  Let us identify $O$ with the point in $\rp^{n-1}$ represented by the  vector 
  $\vec n_\beta(O)=\vec n_\gamma(O)$ and introduce  affine coordinates $(z,y)$, $y=(y_1,\dots,y_{n-2})$, centered at $O$ on an affine chart in $\rp^{n-1}$ so that 
  $H=\{ z=0\}$ and along the curves $\vec n_{\gamma}(x_1)$ and $\vec n_{\beta}(x_1)$ one has $z\simeq x_1$, as 
  $x_1\to0$. One has 
  \begin{equation} z\circ g(\vec n_\gamma(x_1))\simeq z\circ f(\vec n_\gamma(x_1))\simeq 
  -z(\vec n_\gamma(x_1))\simeq -x_1, \ \text{ as } x_1\to0,\label{asx1}\end{equation} 
   since $g$ and $f$ are involutions fixing points in $H$ and the curve $\vec n_\gamma(x_1)$ is transversal 
  to $H$. For every $x_1$ small enough the images 
  $g(\vec n_\gamma(x_1))$, $f(\vec n_\gamma(x_1))$ 
  lie in two  lines $\La_g(x_1)$ and $\La_f(x_1)$ passing through $\vec n_\gamma(x_1)$ and $P_g$  
  (respectively, $\vec n_\gamma(x_1)$ and $P_f$). The difference between the $z$-coordinates of each image 
  and of its source point     is asymptotic to $-2x_1$, by (\ref{asx1}). 
  
  Case 1): $\La_f(0)\neq\La_g(0)$. Then the projections of  the images 
  $g(\vec n_\gamma(x_1))$, $f(\vec n_\gamma(x_1))$ to the hyperplane $H$ along the $z$-axis differ by a quantity of order of $z(\vec n_\gamma(x_1))\simeq x_1$ up to non-zero constant factor. 
  Hence, the images do not differ by $O(x_1^3)$.  
  
  Case 2): $\La_f(0)=\La_g(0)=\La$. We choose the above coordinates $(z,y)$ so that in 
  addition, the line $\La$ is the $z$-axis. 
     Both projective involutions $f$ and $g$ fix the line $\La$, 
  and in the $z$-coordinate   on $\La$ one has 
  \begin{equation} f(z)-g(z)=az^2(1+o(1)), \ \ \text{ as } z\to0; \ \ \ a\neq0,\label{fgz1}
  \end{equation}
  by Proposition \ref{jetneq}. On the other hand, the functions $\Phi_1:=z\circ f$ and 
  $\Phi_2:=z\circ g$ restricted to a small neighborhood of the point $O$ equipped with the 
  coordinates $z=(z,y)$ have difference with 
  the asymptotics 
  \begin{equation} \Phi_2(z)-\Phi_1(z)=\chi(y)z^2+o(z^2), \ \ \text{ as } z\to0,
  \label{cz2}\end{equation}
  where $\chi(y)$ is an analytic function. One has $\chi(0)=a\neq0$, by 
  (\ref{fgz1}). This implies that along the curve $\vec n_\gamma(x_1)$ the functions 
  $\Phi_1$ and $\Phi_2$ differ by a quantity $ax_1^2(1+o(1))\neq O(x_1^3)$. Therefore, 
  the images $f(\vec n_\gamma(x_1))$ and $g(\vec n_\gamma(x_1))$ do not differ by 
  a quantity $O(x_1^3)$. This proves Claim 4.
  \end{proof}
  
  Claim 4 together with (\ref{o3x}) imply that $f=g$. Proposition \ref{pro=} is proved.
  \end{proof}
  
  Let $c$ be the  constant factor in the right-hand side in (\ref{psizg}). Let $P$ be the fixed point 
  of the involution $f=g$ that does not lie in $H$. 
  
  {\bf Claim 5.} {\it Choosing appropriate coordinates $(z,y)$ as at the beginning of the proof of 
  Claim 4 one can achieve that the projection $(z,y)\mapsto y$ is the projection 
  $\rp^{n-1}\setminus\{ P\}\to H$ from the point $P$ and the following asymptotic equalities hold:}
  \begin{equation}z(\vec n_\beta(\zeta))-z(\vec n_\gamma(x_1))=4cx_1^4(1+o(1)),\label{zvn1}
  \end{equation}
  \begin{equation} z(\vec n_\beta(\tilde\zeta))-z(\vec n_\gamma(\hat x_1))=4cx_1^4(1+o(1))
  \ \text{ as } x_1\to0.\label{zvn2}\end{equation}
  \begin{proof} Fix some 
  affine coordinates $y=(y_1,\dots,y_{n-2})$ on $H$ centered at $O$. 
  We extend the coordinate vector function $y$ to all of $\rp^n\setminus\{ P\}$ 
  as the projection $\rp^n\setminus\{ P\}\to H_{y_1,\dots,y_{n-2}}$ 
  from the point $P$.  Let $(x_1,\dots,x_n)$ be the orthogonal coordinates on $\rr^n$ introduced above. For every vector $\vec n\in\rr^n$ set 
$$z(\vec n)=-\frac{x_1(\vec n)}{x_n(\vec n)}.$$
The map $\vec n\mapsto z(\vec n)$ is a projection $\rp^{n-1}\setminus V\to\rp^1_{[x_1:x_n]}$ 
from the codimension two 
projective subspace  $V=\{ x_1=x_n=0\}\subset H\setminus\{ O\}$. (Recall that the latter $O$ is identified with 
$\vec n_\gamma(O)=\vec n_{\beta}(O)$. One has $\vec n_\beta(O)\notin V$, since the vector 
$\vec n_\beta(O)$ is not orthogonal to the $(x_1,x_n)$-coordinate plane $\Pi$: it is orthogonal to the hyperplane 
$T_O\alpha=T_O\Gamma$ transversal to $\Pi$.)  One has $z|_H\equiv 0$, 
by construction, since $H$ is the projectivization of the $(n-1)$-dimensional subspace $\{ x_1=0\}$ 
orthogonal to the line $L$: the $x_1$-axis in the orthogonal coordinates $x_1,\dots,x_n$. 
Afterwards formulas (\ref{zvn1}) and (\ref{zvn2}) follow from formulas (\ref{psizg}) 
and (\ref{psizg'}):  the $z$-coordinates in (\ref{zvn1}) and (\ref{zvn2}) are equal to the tangents 
of   the arguments  in (\ref{psizg}) and (\ref{psizg'}) with added $\frac\pi2$, and $\tan \phi\simeq\phi$, 
as $\phi\to0$.  The claim is proved.
\end{proof}

\begin{proof} {\bf of Lemma  \ref{conjmany2}.} One has $f=g$ (Proposition \ref{pro=}), hence, 
\begin{equation}\vec n_\gamma(\hat x_1)=f(\vec n_\gamma(x_1)), \ 
\vec n_\beta(\tilde\zeta) =f(\vec n_\beta(\zeta)),\label{vfvf}\end{equation}
by definition.

{\bf Claim 6.} {\it Let $(z,y)$ be the coordinates as in Claim 5. Then} 
\begin{equation}z(f(\vec n_\beta(\zeta)))-z(f(\vec n_\gamma(x_1)))=
-(z(\vec n_\beta(\zeta))-z(\vec n_\gamma(x_1)))+O(x_1^5).\label{zff1}\end{equation}

\begin{proof} Each pair "image-preimage" in (\ref{vfvf}) lies in one and the same $z$-fiber, 
since each $z$-fiber is  $f$-invariant (it contains the fixed point $P$ of the involution $f$). The  $y$-coordinates (projections to $H$) of the latter $z$-fibers are $O(x_1^3)$-close to each other: 
$y(\vec n_\gamma(x_1))-y(\vec n_\beta(\zeta))=O(x_1^3)$, by (\ref{ot3}). On the other hand, 
the  restrictions to the 
$z$-fibers of the involution $f$ (written in the coordinate $z$) are of the form $f(z)=-z+\sum_{j=2}^5b_j(y)z^j+o(z^5)$, as $z\to0$, where $b_j(y)$ are analytic functions on a neighborhood of the point $O$.  This and the above statements together imply that 
the restrictions of the projective involution $f$ to the  $z$-fibers containing $\vec n_\gamma(x_1)$ 
and $\vec n_\beta(\zeta)$  differ by a quantity 
$O(z^2x_1^3)$, as $z$ and $x_1$ tend to zero. Let us introduce the auxiliary point 
$q(x_1)$ that lies in the same $z$-fiber, as $\vec n_\gamma(x_1)$, but
 has the same $z$-coordinate, as the other preimage $\vec n_\beta(\zeta)$.  One has 
 \begin{equation}z\circ f(\vec n_\beta(\zeta))-z\circ f(q(x_1))=O(x_1^5),\label{zcf5}\end{equation}
 by the previous statement, since 
$z(\vec n_\gamma(x_1))\simeq z(\vec n_\beta(x_1))=O(x_1)$, and hence, $O(z^2x_1^3)=O(x_1^5)$. 
On the other hand, 
$$z(q(x_1))-z(\vec n_\gamma(x_1))=z(\vec n_\beta(\zeta))-z(\vec n_\gamma(x_1))=
4cx_1^4(1+o(1)),$$
by (\ref{zvn1}). Thus, the points $q(x_1)$ and 
$\vec n_\gamma(x_1)$   lie in the same $z$-fiber and 
have $O(x_1^4)$-close $z$-coordinates that are $O(x_1)$. 
Therefore, the difference of the $z$-coordinates of their $f$-images is equal to minus the difference of their own $z$-coordinates times a quantity $1+O(x_1)$, and $O(x_1)$ times $O(x_1^4)$ yields 
$O(x_1^5)$. This implies (\ref{zff1}) with $\vec n_\beta(\zeta)$ replaced by 
$q(x_1)$. Together with (\ref{zcf5}), this implies (\ref{zff1}). Claim 6 is proved.
\end{proof}

The right-hand side in (\ref{zff1}) is $-4cx_1^4(1+o(1))$, by (\ref{zvn1}). 
But its left-hand side is equal to 
the left-hand side in (\ref{zvn2}), and hence, is 
$4cx_1^4(1+o(1))$, by (\ref{zvn2}). The contradiction thus obtained proves Lemma \ref{conjmany2}.
\end{proof}

Theorem  \ref{conjmany} is proved, since it follows from Lemma \ref{conjmany2}.

\subsection{Proof that the hypersurface is a quadric}

\begin{theorem} \label{tsect} \cite[theorem 4]{art} 
Let a hypersurface in $\rr^n$ be such that each its planar section is a conic. 
Then it is a quadric. 
\end{theorem}

Theorem \ref{conjmany} implies that each planar section of the hypersurface $\alpha$ is 
a conic. This together with Theorem \ref{tsect} proves Theorem \ref{prodim2'}.

\section{Acknowledgements}
I wish to thank Anastasiia Sharipova and Sergei Tabachnikov for introducing me to the area of Hamiltonian, Finsler and projective billiards. I wish to thank them and Vladimir Matveev for helpful discussions. I wish to thank the referee for helpful remarks.


\begin{thebibliography}{}

\bibitem{abate} Abate M.; Patrizio G.  {\it Finsler Metrics -- A Global Approach.} -  Lecture Notes in Mathematics, Volume \textbf{1591}, Springer, Berlin, 1994.

\bibitem{alv1} Alvarez Paiva J.C. {\it Some problems on Finsler geometry.} - Handbook of Differential Geometry 
(editors: Franki J.E. Dillen, Leopold C.A. Verstraelen). North-Holland, Volume 2, 2006, Chapter 1, 1--33. 

\bibitem{alv2} Alvarez Paiva J.C.;  Dur\'an C.E. {\it An introduction to Finsler geometry.} -  Publicaciones de la Escuela Venezolana de Mat\'ematicas, Caracas, Venezuela, 1998.

\bibitem{arn1} Arnold V. {\it Mathematical Methods of Classical Mechanics.} -  Springer, Berlin, 1989.

\bibitem{art} Arnold, M.; Tabachnikov, S. {\it Remarks on Joachimsthal Integral and Poritsky Property.} - 
Arnold Math. J. \textbf{7} (2021), 483--491.

\bibitem{ako1} Artstein-Avidan, S.; Karasev, R.; Ostrover, Y. {\it From symplectic measurements to the Mahler Conjecture.} - Duke Math. J. \textbf{163} (2014), No. 11, 2003--2022.

\bibitem{ao2}  Artstein-Avidan, S.; Ostrover, Y. 
{\it  Bounds for Minkowski billiard trajectories in convex bodies.} -  Int. Math. Res. Not. \textbf{2014} (2014), 
Issue 1,  165--193. 

\bibitem{bao} Bao D.;  Chern S.-S.; Shen Z. {\it An Introduction to Riemann--Finsler Geometry.} -  Springer, Berlin, 2000.



\bibitem{buchin} Buchin, S. {\it Affine Differential Geometry.} - Gordon \& Breach, 
New York, 1983. 

\bibitem{gt} Gutkin, E.; Tabachnikov, S. {\it Billiards in Finsler and Minkowski geometries.} 
-  J. Geom. Phys. \textbf{40} (2002), 277--301. 

\bibitem{hko} Haim-Kislev, P.; Ostrover, Y. {\it A Counterexample to Viterbo’s Conjecture.} - Preprint 
https://arxiv.org/abs/2405.16513

\bibitem{hz} Hofer, H.; Zehnder, E. {\it Symplectic Invariants and Hamiltonian Dynamics.} 
Birkh\"auser, Basel, 1994.

\bibitem{is} Iriyeh, H.; Shibata, M. {\it Symmetric Mahler's conjecture for the volume
product in the 3-dimensional case.} -  Duke Math. J., \textbf{169} (2020),  1077--1134.

\bibitem{schouten}  Schouten, J. A. {\it Ricci-calculus. An introduction to tensor analysis and its geometrical applications.} -  Die Grundlehren der mathematischen Wissenschaften. Vol. \textbf{10} (Second edition of 1923 original ed.). Berlin--G\"ottingen--Heidelberg, Springer-Verlag, 1954. 

\bibitem{tabpr} Tabachnikov, S. {\it Introducing projective billiards.} 
Ergod. Th.  Dynam. Sys. \textbf{17} (1997),  957--976.

\bibitem{viterbo} Viterbo, C. {\it Metric and isoperimetric problems in symplectic geometry.} -  J. Amer. Math. 
Soc. \textbf{13} (2000), 411--431.

\end{thebibliography}
\end{document}